\documentclass[12pt]{amsart} 

\usepackage{amsmath, amssymb, amsfonts, amsbsy, amsthm, array, latexsym, stmaryrd, fullpage, enumerate, ytableau, shuffle, multirow, mathdots, enumerate, tikz,ulem,youngtab,tikz,graphicx}

\usepackage{amsthm}
\usepackage{comment}

\definecolor{mygray}{gray}{0.9}

\newcommand{\tikzmark}[2]{\tikz[overlay,remember picture,baseline] \node [anchor=base] (#1) {$#2$};}

\newcommand{\emptytikzmark}[2]{\tikz[overlay,remember picture,baseline] \node [anchor=base,opacity=0] (#1) {$#2$};}

\newcommand{\DrawVLine}[3][]{%
  \begin{tikzpicture}[overlay,remember picture]
    \draw[shorten >=0.2mm,shorten <=-0.1mm, #1] ([xshift=0.1mm]#2.south east) -- ([xshift=0.1mm]#3.north east);
  \end{tikzpicture}
}

\newcommand{\DrawVLineLeft}[3][]{%
  \begin{tikzpicture}[overlay,remember picture]
    \draw[shorten >=0.03mm,shorten <=-0.1mm, #1] ([xshift=-0.25mm]#2.south west) -- ([xshift=-0.25mm]#3.north west);
  \end{tikzpicture}
}

\newcommand{\DrawHLine}[3][]{%
  \begin{tikzpicture}[overlay,remember picture]
    \draw[shorten <=-0.2mm,shorten >=-0.1mm, #1] (#2.south west) -- (#3.south east);
  \end{tikzpicture}
}

\newcommand{\DrawHLineAbove}[3][]{%
  \begin{tikzpicture}[overlay,remember picture]
    \draw[shorten >=-0.2mm, #1] ([yshift=-0.3mm]#2.north west) -- ([yshift=-0.31mm]#3.north east);
  \end{tikzpicture}
}

\newtheorem{thm}{Theorem}[section]
\newtheorem{lem}[thm]{Lemma}
\newtheorem{prop}[thm]{Proposition}
\newtheorem{cor}[thm]{Corollary}
\newtheorem{conj}[thm]{Conjecture}
\theoremstyle{definition}
\newtheorem{ex}[thm]{Example}
\newtheorem{remark}[thm]{Remark}
\newtheorem{defn}[thm]{Definition}
\newtheorem*{defn*}{Definition}
\newtheorem*{thm*}{Theorem}
\newtheorem*{prop*}{Proposition}
\newtheorem*{cor*}{Corollary}

\newtheorem{conjecture}[thm]{Conjecture}
\newtheorem{question}[thm]{Question}

\usepackage[colorinlistoftodos]{todonotes}

\begin{document}
\title{Coincidences Among Skew Dual Stable Grothendieck Polynomials}
\author{Ethan Alwaise, Shuli Chen, Alexander Clifton, Rebecca Patrias, Rohil Prasad, Madeline Shinners, Albert Zheng}
\maketitle

\begin{abstract}
The question of when two skew Young diagrams produce the same skew Schur function has been well-studied. We investigate the same question in the case of stable Grothendieck polynomials, which are the $K$-theoretic analogues of the Schur functions. We prove a necessary condition for two skew shapes to give rise to the same dual stable Grothendieck polynomial. We also provide a necessary and sufficient condition in the case where the two skew shapes are ribbons.
\end{abstract}

\section{Introduction}
It is well known that the Schur functions indexed by the set of partitions $\{s_{\lambda}\}$ form a linear basis for the ring of symmetric functions over $\mathbb{Z}$. However, for general skew shapes $\lambda / \mu$, the corresponding Schur functions are no longer linearly independent. In fact, two different skew shapes can give rise to the same Schur function. Such skew shapes are called \textit{Schur equivalent}. There are trivial examples of such equivalences---for instance $\langle 2 \rangle$ is clearly Schur-equivalent to $\langle 4 \rangle / \langle 2 \rangle$ as they yield the same shape positioned differently in space---and there are also many non-trivial examples. (Note that we use angled brackets here to denote a partition instead of parentheses to avoid ambiguity with later notation.) For example, the shapes shown below are Schur equivalent \cite{rsw2009coincidences}.
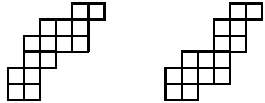
\begin{center}
\ytableausetup{boxsize=.2cm}
\ydiagram{4+2, 2+3,1+4,1+2,2,2}
\qquad\ydiagram{4+2,3+2,3+2,1+3,4,2}
\end{center}

It is natural to ask when these coincidences occur. One application of this type of result involves the representation theory of $GL_N(\mathbb{C})$. In this setting, equality among skew Schur functions corresponds to equivalence of certain $GL_N(\mathbb{C})$ modules \cite{rsw2009coincidences}.
Coincidences among skew Schur functions have been studied by Billera-Thomas-van Willigenburg \cite{billera2006decomposable}, Reiner-Shaw-van Willigenburg \cite{rsw2009coincidences}, and McNamara-van Willigenburg \cite{mcnamara2009towards}, among others. 

The stable and dual stable Grothendieck polynomials are natural ($K$-theoretic) analogues of Schur functions obtained as weighted generating functions over \textit{set-valued tableaux} and \textit{reverse plane partitions}, respectively \cite{buch2002lrrule,lam2007combinatorial}. Roughly speaking, while the Schur functions give infomation about the cohomology of the Grassmannian, these analogues give information about the $K$-theory of the Grassmannian, where $K$-theory is a generalized cohomology theory. Our work concerns the combinatorics of these objects, so knowledge of cohomology theories is not necessary.

The question of coincidences among 
stable and dual stable Grothendieck polynomials of skew shapes was previously unstudied. After a brief background in symmetric functions, we focus on dual stable Grothendieck polynomials of ribbon shape $g_\alpha$, where a ribbon is a connected Young diagram containing no $2\times 2$ square. For a ribbon shape $\alpha$, let $\alpha^*$ denote the shape obtained by 180-degree rotation. We prove the following theorem.

\begin{thm*}[Theorem \ref{thm:littlegribbon}] 
For ribbons $\alpha$ and $\beta$, we have $g_\alpha = g_\beta$ if and only if $\alpha = \beta$ or $\alpha = \beta^*$
\end{thm*}

We next prove two necessary conditions for dual stable Grothendieck equivalence involving \textit{bottleneck numbers} of shape $\lambda/\mu$, $b_i^{\lambda/\mu}$.

\begin{thm*}[Theorem 3.10]
Suppose $g_{\lambda / \mu} = g_{\gamma / \nu}$. Then
\[
 b_i^{\lambda / \mu}+b_{n-i+1}^{\lambda  / \mu} = b_i^{\gamma / \nu}+b_{n-i+1}^{\gamma / \nu}
\]
for $i=1,2,\ldots,n$ where $n$ is the number of columns in $\lambda / \mu$.
\end{thm*}

\begin{thm*}[Corollary \ref{cor:trianglesums}]
Suppose $g_{\lambda/\mu} = g_{\gamma/\nu}$. Then
\[
\sum_{i=1}^n (b_i^{\lambda/\mu})^2 = \sum_{i=1}^n (b_i^{\gamma/\nu})^2.
\]
\end{thm*}

We end by giving examples that show that stable Grothendieck equivalence does not imply dual stable Grothendieck equivalence and vice versa and by highlighting areas for future research.

\section{Preliminaries}
\subsection{Partitions and tableaux}
A \textit{partition} $\lambda=\langle \lambda_1,\lambda_2,\ldots,\lambda_k \rangle$ of a positive integer $n$ is a weakly decreasing sequence of positive integers $\lambda_1 \geq \lambda_2 \geq \cdots \geq \lambda_k > 0$ whose sum is $n$. The integer $\lambda_i$ is called the $i$th \textit{part} of $\lambda$. We call $n$ the \textit{size} of $\lambda$, denoted by $\vert \lambda \vert = n$. Throughout this document $\lambda$ will refer to a partition. We may visualize a partition $\lambda$ using a \textit{Young diagram}: a collection of left-justified boxes where the $i$th row from the top has $\lambda_i$ boxes. For example, the Young diagram of $\lambda = \langle 5,2,1,1 \rangle$ is shown below.

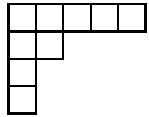
\ytableausetup{boxsize=.35cm}
$$\ydiagram{5,2,1,1}$$

A \textit{skew shape} $\lambda/\mu$ is a pair of partitions $\lambda = \langle\lambda_1,\ldots,\lambda_m \rangle$ and $\mu = \langle\mu_1,\ldots,\mu_k\rangle$ such that $k \leq m$ and $\mu_i \leq \lambda_i$ for all $i$. We form the Young diagram of a skew shape $\lambda / \mu$ by superimposing the Young diagrams of $\lambda$ and $\mu$ and removing the boxes that are contained in both. If $\mu$ is empty, $\lambda/\mu=\lambda$ is called a \textit{straight shape}. Given a skew shape $\lambda/\mu$, we define its \textit{antipodal rotation} $(\lambda/\mu)^*$ as the skew shape obtained by rotating the Young diagram of $\lambda/\mu$ by $180$ degrees. For example, the Young diagrams of the skew shapes $\langle 6,3,1 \rangle / \langle 3,1\rangle $ and $(\langle 6,3,1\rangle / \langle 3,1\rangle)^*$ are shown below.

$$\langle 6,3,1 \rangle /\langle 3,1\rangle =\ydiagram{3+3, 1+2, 1} \qquad \qquad (\langle 6,3,1 \rangle /\langle 3,1 \rangle)^*=\ydiagram{5+1, 3+2, 3}$$

A \textit{semistandard Young tableau} of shape $\lambda / \mu$ is a filling of the boxes of the Young diagram of $\lambda / \mu$ with positive integers such that the entries weakly increase from left to right across rows and strictly increase from top to bottom down columns. Two semistandard Young tableaux are shown below.
\ytableausetup{boxsize=.45cm}
$$\begin{ytableau}1 & 1 & 4 & 7 \\ 2 & 6 \\ 9\end{ytableau}\qquad\qquad \begin{ytableau}\none & \none & 1 & 3 & 3 \\ \none & 1 & 4 & 6 \\ 1 & 4\end{ytableau}$$

A \textit{set-valued tableau} of shape $\lambda / \mu$ is a filling of the boxes of the Young diagram of $\lambda / \mu$ with finite, nonempty sets of positive integers such that the entries weakly increase from left to right across rows and strictly increase from top to bottom down columns.  For two sets of positive integers $A$ and $B$, we say that $A \leq B$ if $\max{A} \leq \min{B}$ and $A < B$ if $\max{A} < \min{B}$. For a set-valued tableau $T$, we define $\vert T \vert$, the \textit{size} of $T$, to be the sum of the sizes of the sets appearing as entries in $T$. For example,
\ytableausetup{boxsize=.75cm} 
$$\begin{ytableau} 1,2 & 2,3 & 6 & 9 \\ 
3 & 5 \\
6 & 6,7 \end{ytableau}$$
is a set-valued tableau of shape $\lambda = \langle 4,2,2 \rangle $ and size $11$.

A \textit{reverse plane partition} (RPP) of shape $\lambda / \mu$ is a filling of the boxes of the Young diagram of $\lambda / \mu$ with positive integers such that the entries weakly increase both from left to right across rows and from top to bottom down columns. For example,
\ytableausetup{boxsize=.45cm}
$$\begin{ytableau}\none & 1 & 1 & 2 & 7 \\ \none & 1 & 2 & 2 & 8 \\ 1 & 2 & 2 & 2\end{ytableau}$$
is a reverse plane partition of shape $\langle 5,5,4\rangle /\langle 1,1 \rangle $.
\subsection{Symmetric functions}
To each of the above fillings of a Young diagram we may associate a monomial as follows. First, let $T$ be a semistandard or set-valued tableau. We associate a monomial $x^T$ given by
$$x^T = \prod_{i \in \mathbb{N}} x_i^{m_i},$$
where $m_i$ is the number of times the integer $i$ appears as an entry in $T$.
For example, the semistandard Young tableaux shown above correspond to monomials $x_1^2x_2x_4x_6x_7x_9$ and $x_1^3x_3^2x_4^2x_6$, respectively, while the set-valued tableau corresponds to monomial $x_1x_2^2x_3^2x_5x_6^3x_7x_9$.

Given a reverse plane partition $P$, the associated monomial $x^P$ is given by
$$x^P = \prod_{i \in \mathbb{N}} x_i^{m_i},$$
where $m_i$ is the number of columns of $P$ that contain the integer $i$ as an entry. The reverse plane partition shown above has monomial $x_1^3x_2^3x_7x_8$.

We can now define the Schur functions, the stable Grothendieck polynomials, and the dual stable Grothendieck polynomials, which are all indexed by skew shapes.

We define the \textit{Schur function} $s_{\lambda / \mu}$ by
$$s_{\lambda / \mu} = \sum_{T}^{}x^T,$$
where we sum over all semistandard Young tableaux of shape $\lambda / \mu$. Note that entries may be any positive integer, so $s_{\lambda/\mu}$ will be an infinite sum where each term has degree $|\lambda/\mu|=|\lambda|-|\mu|$. For example, 
\[s_{\langle 1 \rangle}=x_1+x_2+x_3+x_4+\ldots,\] and 
\[s_{\langle 2,1\rangle}=x_1^2x_2+x_1^2x_3+x_2^2x_3+\ldots + 2x_1x_2x_3+2x_1x_2x_4+\ldots+2x_4x_8x_{101}+\ldots.\]

Though is it not obvious from this combinatorial definition, the Schur functions are symmetric functions. In other words, each $s_{\lambda/\mu}$ is unchanged after permuting any finite subset of the infinite variable set $\{x_1,x_2,\ldots\}$. Moreover, the Schur functions indexed by straight shapes $\{s_\lambda\}$ form a basis for the ring of symmetric functions over $\mathbb{Z}$. These functions arise naturally in areas like algebraic combinatorics, representation theory, and Schubert calculus. We refer the interested reader to \cite{EC2} for further reading on Schur functions and symmetric functions.  

We next define the \textit{stable Grothendieck polynomial}, the first of two \textit{$K$-theoretic analogues} of the Schur functions. We direct the interested reader to \cite{buch2002lrrule} for more on this topic and for an explanation of the connection to $K$-theory. The stable Grothendieck polynomial $G_{\lambda / \mu}$ is defined by
$$G_{\lambda / \mu} = \sum_{T}^{}(-1)^{\vert T \vert -\vert \lambda \vert}x^T,$$
where we sum over all set-valued tableaux of shape $\lambda / \mu$. 

Note that semistandard tableaux are set-valued tableaux where each subset has size one. It follows that each $G_{\lambda/\mu}$ will be a sum of $s_{\lambda/\mu}$ plus terms of degree greater than $|\lambda/\mu|$. While each term in a Schur function has the same degree, each stable Grothendieck polynomial is an infinite sum where terms have arbitrarily large degree. For example, 
\[G_{\langle 1 \rangle}=x_1+x_2+\ldots -x_1x_2-x_2x_3+\ldots+x_1x_2x_4x_5x_9+\ldots,\] and \[G_{\langle 2,2 \rangle/\langle 1\rangle}=x_1^2x_2+2x_1x_2x_3+\ldots- 3x_1^2x_2x_3-8x_2x_5x_9x_{114}-\ldots+2x_1^2x_2^2x_3+\ldots.\]


The other natural $K$-theoretic analogue of the Schur function is the \textit{dual stable Grothendieck polynomial}. It is dual to the stable Grothendieck polynomial under the Hall inner product. We refer the reader to \cite{lam2007combinatorial} for more background. We define the dual stable Grothendieck polynomial $g_{\lambda / \mu}$ by
$$g_{\lambda / \mu} = \sum_{P}^{}x^P,$$
where the sum is over all reverse plane partitions of shape $\lambda / \mu$.

Again, note that semistandard Young tableaux are examples of reverse plane partitions where the columns are strictly increasing. As a result, each dual stable Grothendieck polynomial $g_{\lambda/\mu}$ is a sum of the Schur function indexed by the same shape $s_{\lambda/\mu}$ and terms of degree strictly less than $|\lambda/\mu|$. They are again infinite sums, but now each term has degree at most $|\lambda/\mu|$ and at least the number of columns in shape $\lambda/\mu$. For example,
\[g_{\langle 2,1\rangle}=x_1^2x_2+2x_1x_2x_3+\ldots+x_1x_2+x_1x_3+\ldots+x_1^2+x_2^2+\ldots.\]

Though it is again not obvious from the definitions, both the stable and dual stable Grothendieck polynomials are symmetric functions. We use this fact throughout this paper.

We say that two skew shapes $D_1$ and $D_2$ are $G$-\textit{equivalent} or $g$-\textit{equivalent} if $G_{D_1} = G_{D_2}$ or $g_{D_1} = g_{D_2}$, respectively. Since any $G_D$ contains $s_D$ as its lowest degree terms, $G_{D_1}=G_{D_2}$ implies $s_{D_1}=s_{D_2}$. Similarly, $g_{D_1}=g_{D_2}$ implies $s_{D_1}=s_{D_2}$.
Furthermore, it is straightforward to check that two skew shapes that are equivalent in any of the three aforementioned senses must have the same number of rows and columns. We will implicitly use this fact throughout.

It is an easy consequence of symmetry that all three notions of skew equivalence are preserved under antipodal rotation, $^*$. We provide a proof for stable Grothendieck polynomials below.

\begin{prop}
For any skew shape $\lambda/\mu$, $G_{\lambda/\mu} = G_{(\lambda/\mu)^*}$ and $g_{\lambda/\mu} = g_{(\lambda/\mu)^*}$.
\end{prop}
\begin{proof}
We prove the result for stable Grothendieck polynomials; the argument for dual stable Grothendieck polynomials is similar.
Let $x_I = x_{i_1}^{p_1}x_{i_2}^{p_2}\ldots x_{i_k}^{p_k}$ be a monomial with $i_1 < i_2 < \ldots < i_k$. It suffices to show that the $x_I$-coefficient of each of the two polynomials is equal. To do so, we construct a bijection between set-valued tableaux of shape $\lambda/\mu$ with weight monomial $x_I$ and set-valued tableaux of shape $(\lambda/\mu)^*$ with weight monomial $x_{I'} = x_{i_k+1-i_1}^{p_1}x_{i_k+1-i_2}^{p_2} \ldots x_{1}^{p_k}$. This bijection, which is in fact an involution, maps a tableau $T$ to the tableau $T'$ given by rotating $T$ and then replacing every entry $j$ with $i_k + 1 - j$. An example is given below where $i_k = 5$.
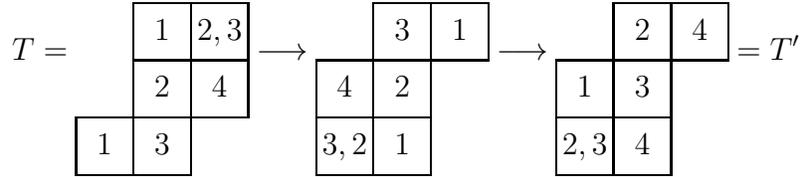
\ytableausetup{boxsize=.75cm}
$$T=\begin{ytableau} \none & 1 & 2, 3 \\ \none & 2 & 4 \\ 1 & 3 \end{ytableau} \longrightarrow \begin{ytableau} \none & 3 & 1 \\ 4 & 2 \\ 3, 2 & 1 \end{ytableau} \longrightarrow \begin{ytableau} \none & 2 & 4 \\ 1 & 3 \\ 2, 3 & 4 \end{ytableau}=T'$$

Thus, the $x_{I'}$-coefficient of $G_{(\lambda/\mu)^*}$ is equal to the $x_I$-coefficient of $G_{\lambda/\mu}$. By symmetry, the $x_{I'}$-coefficient of $G_{(\lambda/\mu)^*}$ is equal to the $x_I$-coefficient of $G_{(\lambda/\mu)^*}$, so the $x_I$-coefficients of $G_{\lambda/\mu}$ and $G_{(\lambda/\mu)^*}$ are equal as desired. 
\end{proof}

\subsection{Ribbon shapes}

We will be interested in a special class of skew shapes known as \textit{ribbons}. A skew shape $\alpha$ is called a ribbon if it is connected and contains no $2 \times 2$ rectangle. Being connected means that if there is more than one box, then each box must share an edge with another box. The shape shown below on the left is a ribbon while the shape in the middle and on the right are not. The shape in the middle contains a $2\times 2$ rectangle and the shape on the right is not connected. 
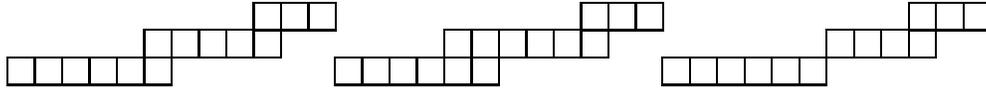
\ytableausetup{boxsize=.35cm}

$$\ydiagram{9+3,5+5,6}\ydiagram{9+3,4+6,6}\ydiagram{9+3,6+4,6}$$

A \textit{composition} of a positive integer $n$ is an ordered list of positive integers that sum to $n$. We will write compositions inside of parentheses. For example, $(2,7,4,9)$ is a composition of $22$. It is easy to see that ribbons of size $n$ are in bijection with compositions of $n$: to obtain a composition from a ribbon shape, simply read the row sizes from bottom to top. This is clearly a bijection. For this reason, we will denote a ribbon shape by the associated composition $\alpha$. For example, the we denote the ribbon shown above by $(6,5,3)$. 


Note that one can also construct a bijection between compositions and ribbons using the sizes of the columns of $\alpha$ read from left to right. We also describe ribbon shapes this way, and we use square brackets in place of parentheses to denote this column reading. For example, the ribbon shown above may be written as $[1,1,1,1,1,2,1,1,1,2,1,1]$.


Notice that the antipodal rotation $\alpha^{*}$ of $\alpha=(\alpha_1, \alpha_2,\ldots,\alpha_k)$ is the ribbon  $(\alpha_k,\alpha_{k-1},\ldots,\alpha_1)$. We refer to $\alpha^*$ as the \textit{reverse ribbon} of $\alpha$. For a ribbon shape $\alpha$, we denote the corresponding Schur function by $s_\alpha$ and refer to it as the \textit{ribbon Schur function.} 

We now define several binary operations on the set of ribbons as in \cite{rsw2009coincidences}. 
Here we let $\alpha = (\alpha_1,\ldots,\alpha_k)$ and $\beta = (\beta_1,\ldots,\beta_m)$ be ribbons. We define the concatenation operation

$$\alpha \cdot \beta = (\alpha_1,\ldots,\alpha_k,\beta_1.\ldots,\beta_m)$$
and the near concatenation operation
$$\alpha \odot \beta = (\alpha_1,\ldots,\alpha_{k-1},\alpha_k + \beta_1,\beta_2,\ldots,\beta_m).$$
We let
$$\alpha^{\odot n} = \underbrace{\alpha \odot \cdots \odot \alpha}_{\text{$n$}}.$$
We can combine the two concatenation operations to yield a third operation $\circ$, defined by
$$\alpha \circ \beta = \beta^{\odot\alpha_1}\cdot \beta^{\odot\alpha_2} {\cdots} \beta^{\odot\alpha_k}.$$

\begin{ex}
Consider ribbons $\alpha=(3,2)$ and $\beta=(1,2)$ shown below. 
$$\alpha=\ydiagram{2+2,3} \qquad \beta=\ydiagram{2,1}$$
Then $\alpha\cdot\beta$ and $\alpha\odot\beta$ are as follows.
$$\alpha\cdot\beta=\begin{ytableau}\none & \none & \none & *(gray) & *(gray) \\ \none & \none & \none & *(gray) \\ \none & \none & & \\ $ $ & &  \end{ytableau}\qquad \alpha\odot \beta = \begin{ytableau}\none & \none & \none & \none & *(gray) & *(gray)\\ \none & \none & & & *(gray)\\ $ $ & & \end{ytableau}$$
The operation $\alpha\circ\beta$ replaces each square of $\alpha$ with a copy of $\beta$. The copies of $\beta$ are near-concatenated if the corresponding blocks of $\alpha$ are horizontally adjacent and concatenated if the corresponding blocks of $\alpha$ are vertically adjacent. 

$$\alpha\circ\beta=\begin{ytableau}\none & \none & \none & \none & \none & \none & \none & *(gray)& *(gray)\\ \none & \none & \none & \none & \none & & & *(gray) \\ \none & \none & \none & \none & \none & \\ \none & \none & \none & \none &*(gray) & *(gray)\\ \none & \none & & & *(gray) \\ *(gray) $ $ & *(gray)&  \\ *(gray) $ $ \end{ytableau}$$

\end{ex}

If a ribbon $\alpha$ can be written in the form $\alpha = \beta_1 \circ \cdots \circ \beta_{\ell}$, we call this a \textit{factorization} of $\alpha$. A factorization $\alpha = \beta \circ \gamma$ is called \textit{trivial} if any of  the following conditions hold:
\begin{enumerate}
\item one of $\beta$ or $\gamma$ consists of a single square,
\item both $\beta$ and $\gamma$ consist of a single row, or
\item both $\beta$ and $\gamma$ consist of a single column.
\end{enumerate}

A factorization $\alpha = \beta_1 \circ \cdots \circ \beta_{\ell}$ is called \textit{irreducible} if none of the factorizations $\beta_i \circ \beta_{i+1}$ are trivial and each $\beta_i$ has no nontrivial factorization. In \cite{billera2006decomposable}, the authors prove that every ribbon $\alpha$ has a unique irreducible factorization. They then prove the following theorem:

\begin{thm}\cite{billera2006decomposable}\label{irred_fact}
Two ribbons $\alpha$ and $\beta$ satisfy $s_{\alpha} = s_{\beta}$ if and only if $\alpha$ and $\beta$ have irreducible factorizations
$$\alpha = \alpha_1 \circ \cdots \circ \alpha_k \quad \textrm{and} \quad \beta = \beta_1 \circ \cdots \circ \beta_k,$$
where each $\beta_i$ is equal to either $\alpha_i$ or $\alpha_i^{*}$.
\end{thm}

In the next section, we use the above theorem to prove a necessary and sufficient condition for two ribbons to be $g$-equivalent. We also provide a necessary condition for two skew shapes to be $g$-equivalent.

\section{Coincidences of Dual Stable Grothendieck Polynomials}

\subsection{Ribbons}
The main result of this section is that for two ribbons $\alpha$ and $\beta$, $g_\alpha = g_\beta$ if and only if $\alpha = \beta$ or $\alpha = \beta^*$. We will obtain restrictions on $\alpha$ and $\beta$ by writing the dual stable Grothendieck polynomials in terms of ribbon Schur functions and comparing the coefficients in the resulting expansions.

The next proposition requires the following ordering on ribbons. 
For ribbons $\alpha = [\alpha_1,\ldots,\alpha_n]$ and $\gamma=[\gamma_1,\ldots,\gamma_n]$ with the same number of columns, we write $\gamma\leq{\alpha}$ if $\gamma_i\leq{\alpha_i}$ for each $i = 1, \ldots, n$.

\begin{prop} \label{ribbon_expand}
Let $\alpha = [\alpha_1,\ldots,\alpha_n]$ be a ribbon. The dual stable Grothendieck polynomial $g_\alpha$ can be decomposed into a sum of ribbon Schur functions as
\[
g_\alpha = \sum_{\gamma \leq \alpha} \left( \prod_{i=1}^{n} \binom{\alpha_i-1}{\alpha_i-\gamma_i} \right) s_\gamma .
\]
\end{prop}

\begin{proof}
We define a map from reverse plane partitions of ribbon shape $\alpha$ to the set of semistandard Young tableaux of shape $\gamma$ where $\gamma \leq \alpha$. Given a reverse plane partition $T$ of shape $\alpha$, map $T$ to a semistandard Young tableau of shape $\gamma = [\gamma_1,\ldots.\gamma_n]$ where $\gamma_i$ is the number of distinct entries in column $i$ in $T$. Fill column $i$ of $\gamma$ with the distinct entries of column $i$ in $T$ in increasing order. This gives a semistandard Young tableau because columns are clearly strictly increasing and rows will remain weakly increasing. 

This map preserves the monomial corresponding to the reverse plane partition. The map is also surjective, since any semistandard Young tableau of shape $\gamma$ where $\gamma \leq \alpha$ is mapped to by any reverse plane partition with the same entries in each column but with some entries copied.

It remains to show each semistandard Young tableau is mapped to by exactly $\prod \binom{\alpha_i-1}{\alpha_i-\gamma_i}$ reverse plane partitions. Fix some semistandard Young tableau of shape $\gamma \leq \alpha$. We construct all possible reverse plane partitions of $\alpha$ mapping to this semistandard Young tableau column by column. Given column $i$ of $\alpha$, consider the $\alpha_i-1$ pairs of adjacent squares in the column. Since there are $\gamma_i$ distinct entries in the column and the entries are written in weakly increasing order, $\alpha_i - \gamma_i$ of these pairs must match. A size $(\alpha_i-\gamma_i)$ subset of the $\alpha_i-1$ pairs of adjacent squares gives a unique filling, where the given subset is the set of adjacent squares that match. Thus the number of possible fillings for each column is $\binom{\alpha_i-1}{\alpha_i-\gamma_i}$, giving the desired formula. 
\end{proof}

\begin{lem}
\label{ribbon-col-sum}
Let $\alpha = [\alpha_1, \ldots, \alpha_n]$ and $\beta = [\beta_1, \ldots, \beta_n]$ be ribbons such that $g_\alpha = g_\beta$. Then for all $i=1,\ldots,n$ we have $\alpha_i + \alpha_{n-i+1} = \beta_i + \beta_{n-i+1}$.
\end{lem}
\begin{proof}
Use Proposition \ref{ribbon_expand} to write $g_\alpha$ and $g_\beta$ as a sum of ribbon Schur functions. Note that all terms of degree $n+1$ in both sums are of the form $s_\gamma$ where $\gamma$ is a ribbon $(i,n-i+1)$. Let $A$ denote the set of all compositions of $n+1$ with weakly decreasing parts (i.e. the set of partitions of $n+1$). It is shown in Proposition 2.2 of \cite{billera2006decomposable} that the set $\lbrace{s_\alpha}\rbrace_{\alpha\in A}$ forms a basis for $\Lambda_{n+1}$, the degree $n+1$ elements of the ring of symmetric functions. Then since each ribbon $(i,n-i+1)$ is Schur equivalent to $(n-i+1,i)$, it follows that the set of Schur functions of such ribbons is linearly independent. Comparing coefficients in the respective sums gives the desired equality.
\end{proof}

\begin{lem}\label{factor_sym}
Suppose $\alpha$ and $\beta$ are ribbons such that $g_{\alpha} = g_{\beta}$, $\alpha \neq \beta$, and there exist ribbons $\sigma$, $\tau$, and $\mu$ such that $\alpha = \sigma \circ \mu$ and $\beta =  \tau \circ \mu$. Then $\mu = \mu^*$.
\end{lem}
\begin{proof}
Let $\mu = [\mu_1,\ldots,\mu_t]$, $\alpha = [\alpha_1,\ldots,\alpha_n]$, and $\beta = [\beta_1,\ldots,\beta_n]$. By hypothesis, we have that $\alpha = \mu \square_1 \cdots \square_r \mu$ and $\beta = \mu \diamond_1 \cdots \diamond_s \mu$, where each $\square_i$ and $\diamond_i$ is one of the operations $\cdot$ or $\odot$.  Thus each $\alpha_i$ and $\beta_i$ is equal to one of $\mu_1,\ldots,\mu_t$ or $\mu_1 + \mu_t$. Since $\alpha \neq \beta$, let $r$ be the minimal index such that $\alpha_r \neq \beta_r$. We see that $\{\alpha_r,\beta_r\} = \{\mu_t,\mu_1 + \mu_t \}$ because the first index where $\alpha$ and $\beta$ disagree corresponds to the first index $i$ where $\square_i \neq \diamond_i$. By Lemma \ref{ribbon-col-sum} it follows that if $\alpha_i = \beta_i$ then $\alpha_{n-i+1} = \beta_{n-i+1}$. Hence $n-r+1$ is the maximal index where $\alpha$ and $\beta$ disagree. Note that by the same argument we similarly have $\{\alpha_{n-r+1},\beta_{n-r+1}\} = \{\mu_1,\mu_1 + \mu_t \}$. 

We have $\alpha_r + \alpha_{n-r+1} = \beta_r + \beta_{n-r+1}$ by Lemma \ref{ribbon-col-sum}. Substituting the possible values of $\alpha_r \neq \beta_r$ and $\alpha_{n-r+1} \neq \beta_{n-r+1}$, we find that this equation is either
$$\mu_1 + \mu_t = 2(\mu_1 + \mu_t)$$
or
$$2\mu_1 + \mu_t = \mu_1 + 2\mu_t.$$
The first equation is a contradiction. Thus the second equation holds, implying that $\mu_1 = \mu_t$. We will show by induction that $\mu_i = \mu_{t - i + 1}$, completing the proof. We have just shown the base case. For the general case, we have by Lemma \ref{ribbon-col-sum}
$$\alpha_{r + i} + \alpha_{n - r - i + 1} = \beta_{r + i} + \beta_{n - r - i + 1}.$$
We may assume without loss of generality that $\alpha_r = \mu_t$. Then we have $\alpha_{n - r + 1} = \mu_1 + \mu_t$, $\beta_r = \mu_1 + \mu_t$ and $\beta_{n - r + 1} = \mu_1$. Therefore 
\begin{equation*}
\begin{aligned}
\alpha_{r + i} &= \mu_i \\
\beta_{r + i} &= \mu_{i + 1} \\
\alpha_{n - r - i + 1} &= \mu_{t - i} \\
\beta_{n - r - i + 1} &= \mu_{t - i + 1}.
\end{aligned}
\end{equation*}
We thus have
$$\mu_i + \mu_{t - i} = \mu_{i + 1} + \mu_{t - i +1}.$$
By the inductive hypothesis $\mu_i = \mu_{t - i + 1}$, so $\mu_{i + 1} = \mu_{t - i}$, finishing the proof.
\end{proof}

We are now ready for the main result of this section. 

\begin{thm} \label{thm:littlegribbon}
For ribbons $\alpha$ and $\beta$, $g_\alpha = g_\beta$ if and only if $\alpha = \beta$ or $\alpha = \beta^*$.
\end{thm}
\begin{proof}
Suppose $g_\alpha = g_\beta$. Then $s_\alpha = s_\beta$. By Theorem \ref{irred_fact}, we have irreducible factorizations 
\[
\alpha = \alpha_k \circ \cdots \circ \alpha_1
\]
\[
\beta = \beta_k \circ \cdots \circ \beta_1.
\] 
(We reverse the indices for ease of induction.) We prove by induction on $r$ that for $r = 1, \ldots, k$ we have
\[
\alpha_r \circ \cdots \circ \alpha_1 \in \{ \beta_r \circ \cdots \circ \beta_1, (\beta_r \circ \cdots \circ \beta_1)^* \}.
\]
By Theorem \ref{irred_fact} we have $\alpha_1 \in \{ \beta_1, \beta_1^* \}$ so the base case is satisfied. Now suppose $r \geq 2$. By the inductive hypothesis we have 
\[
\alpha_{r-1} \circ \cdots \circ \alpha_1 \in \{ \beta_{r-1} \circ \cdots \circ \beta_1, (\beta_{r-1} \circ \cdots \circ \beta_1)^* \}.
\]
If $\alpha = \beta$ we are done, so we may assume otherwise. Then by letting $\mu = \alpha_{r-1} \circ \cdots \circ \alpha_1$ and applying Lemma \ref{factor_sym} to $\alpha$ and either $\beta$ or $\beta^*$, we have
\[
\beta_{r-1} \circ \cdots \circ \beta_1 = (\beta_{r-1} \circ \cdots \circ \beta_1)^*.
\]
Since we also have that $\alpha_r \in \{ \beta_r, \beta_r^* \}$ we are done.
\end{proof}

\subsection{Necessary Condition: Bottlenecks}

We now move to the case of determining equality of dual stable Grothendieck polynomials of general skew shape. We introduce the ``bottleneck numbers'' of a skew diagram and use these to construct closed-form expressions for certain coefficients of its dual Grothendieck polynomial. We then obtain a necessary condition for $g$-equivalence that generalizes Lemma \ref{ribbon-col-sum}. 

For the following definition, we define an \textit{interior horizontal edge} to be a horizontal edge of a box in a Young diagram that lies neither at the top boundary nor the bottom boundary of the Young diagram.

\begin{defn}
A \textit{bottleneck edge} in a skew shape $\lambda/\mu$ is an interior horizontal edge touching both the left and right boundaries of the shape. For example, the red edges in Figure \ref{bottleneck_ex} are bottleneck edges. We let $b_i^{\lambda/\mu}$ denote the number of bottleneck edges in column $i$. 
\end{defn}

If the shape $\lambda/\mu$ has $n$ columns and $m$ rows, then number of bottleneck edges in column $i$ for $i = 1,2, \ldots , n$ is equivalently
\[
 b_i^{\lambda / \mu} = | \lbrace 1\leq{j}\leq{m-1} \mid \mu_j = i-1, \lambda_{j+1}=i \rbrace |.
\]
When the skew shape in question is clear, we will often suppress the superscript.

Bottleneck edges are related to the \textit{row overlap compositions} defined in \cite{rsw2009coincidences}, which we now review.

\begin{defn}[\cite{rsw2009coincidences}]
The $k$-\textit{row overlap composition} $r^{(k)}$ of a skew diagram $\lambda/\mu$ with $m$ rows is $(r^{(k)}_1,\ldots,r^{(k)}_{m-k+1})$, where $r^{(k)}_i$ is the number of columns containing squares in all the rows $i,i+1,\ldots,i+k-1$. 
\end{defn}

In particular, $r^{(2)} = (\lambda_2-\mu_1,\lambda_{3}-\mu_2,\ldots,\lambda_m - \mu_{m-1})$. Thus bottleneck edges correspond to 1's in the 2-row overlap composition. When the $2, 3, \ldots, m$ row overlap compositions are written, they form a triangular array of non-negative integers as shown in Example \ref{ex:rowoverlap}. A column having $i$ bottleneck edges corresponds in the array to an equilateral triangle of 1's with side length $i$. In \cite{rsw2009coincidences}, it is proven that the $k$-overlap compositions of two Schur equivalent shapes are permutations of each other for each $k$.

\begin{ex}\label{ex:rowoverlap}
Let $\lambda / \mu = \langle 5,5,4,2,2,2\rangle /\langle 4,2,1,1,1\rangle$.
\begin{figure}[h]
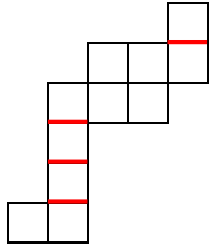

\ytableausetup
{boxsize=1.25em}
\ytableausetup
{aligntableaux=top}
\begin{ytableau}
\none & \none & \none & \none & \emptytikzmark{6}{1} \\
\none & \none & \emptytikzmark{4}{1}  & \emptytikzmark{5}{1}  &  \\
\none & \emptytikzmark{3}{1} &   &  \\
\none & \emptytikzmark{2}{1}    \\
\none & \emptytikzmark{1}{1}    \\
      &   \\
\end{ytableau}

\DrawHLine[red, ultra thick]{1}{1}
\DrawHLine[red, ultra thick]{2}{2}
\DrawHLine[red, ultra thick]{3}{3}
\DrawHLine[red, ultra thick]{6}{6}

\caption{$\langle 5,5,4,2,2,2 \rangle / \langle 4,2,1,1,1 \rangle $ has 3 bottleneck edges in column 2 and 1 bottleneck edge in column 5.}\label{bottleneck_ex}
\end{figure}
Then the number of bottleneck edges in each column is shown below. $(b_1,b_2,b_3,b_4,b_5)=(0,3,0,0,1)$. The row overlap compositions $r^{(2)},\ldots,r^{(6)}$ are
\[\begin{tabular}{>{$}l<{$\hspace{12pt}}*{13}{c}}
r^{(6)} &&&&&&&0&&&&&&\\
r^{(5)} &&&&&&0&&0&&&&&\\
r^{(4)} &&&&&0&&0&&1&&&&\\
r^{(3)} &&&&0&&0&&1&&1&&&\\
r^{(2)} &&&1&&2&&1&&1&&1&&\\
\end{tabular}\]
\end{ex}

\begin{defn}
We define a \textit{1,2-RPP} to be a reverse plane partition involving only 1's and 2's. A \textit{mixed} column of a 1,2-RPP contains both 1's and 2's while an \textit{i-pure} column contains only $i$'s.
\end{defn}

Note the 1,2-RPP's of a given shape are in bijection with lattice paths from the upper right vertex of the shape to the lower left vertex of the shape. The corresponding 1,2-RPP can be generated from such a lattice path by filling the squares below the path with 2's and the squares above the path with 1's. Conversely, the corresponding lattice path can be recovered from a 1,2-RPP by drawing horizontal segments below the last 1 (if there are any) in a column and above the first 2 (if there are any) in a column. Vertical segments can then be drawn to connect these horizontal segments into a lattice path. Observe that mixed columns in the 1,2-RPP correspond to interior horizontal edges in the lattice path.

\begin{figure}[h]
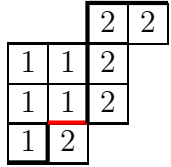

\ytableausetup
{boxsize=1.25em}
\ytableausetup
{aligntableaux=top}
\begin{ytableau}
\none & \none & \tikzmark{r4c3}{2} & \tikzmark{r4c4}{2}\\
1 & \tikzmark{r3c2}{1} & \tikzmark{r3c3}{2}\\
1 & \tikzmark{r2c2}{1} & \tikzmark{r2c3}{2}\\
\tikzmark{r1c1}{1} & 2
\end{ytableau}

\DrawHLine[black, ultra thick]{r1c1}{r1c1}
\DrawVLine[black, ultra thick]{r1c1}{r1c1}
\DrawHLine[red, ultra thick]{r2c2}{r2c2}
\DrawVLineLeft[black, ultra thick]{r2c3}{r4c3}
\DrawHLineAbove[black, ultra thick]{r4c3}{r4c4}

\caption{1,2-RPPs correspond to lattice paths inside the skew shape. Note the red interior horizontal edge corresponds to the boundary between the 1's and 2's in the mixed column.}
\end{figure}

\begin{thm} \label{bottleneck_cond} Let $\lambda/\mu$ be a skew shape with $n$ columns, and suppose $g_{\lambda / \mu} = g_{\gamma / \nu}$. Then
\[
 b_i^{\lambda / \mu}+b_{n-i+1}^{\lambda  / \mu} = b_i^{\gamma / \nu}+b_{n-i+1}^{\gamma / \nu}
\]
for $i=1,2,\ldots,n$. 
\end{thm}
Note that $\gamma/\nu$ must also have $n$ columns.
\begin{proof}
Fix a shape $\lambda / \mu$ with $m$ rows and $n$ columns. We will compute the coefficients for terms of the form $x_1^r x_2^{n-r+1}$ in $g_{\lambda/\mu}$. Since $g_{\lambda/\mu}$ is symmetric, we may assume without loss of generality that $r \leq n-r+1$. 

By the bijection between 1,2-RPP's and lattice paths given above, we may compute the coefficient of $x_1^r x_2^{n-r+1}$ by counting the number of lattice paths corresponding to this monomial. Note that any such lattice path must have exactly one interior horizontal edge. For each interior horizontal edge, we will count the number of lattice paths corresponding to the monomial $x_1^r x_2^{n-r+1}$ using the given edge. There are four cases: the interior horizontal edge either touches neither boundary, only the left boundary, only the right boundary, or both the left and right boundary (i.e. the edge is a bottleneck edge).

Fix an interior horizontal edge and suppose it lies in column $i$. Consider first the case where the interior horizontal edge touches neither boundary. Suppose a lattice path uses the given edge as its only interior horizontal edge. Then, as depicted in Figure \ref{no_boundary}, the lattice path must travel the top boundary until column $i$ and then drop to the horizontal edge. Then from the left endpoint of the given edge the path must drop to the bottom boundary and travel along the bottom boundary until reaching the bottom left. Hence there is a unique lattice path that uses the given edge as its only interior horizontal edge. Note that the corresponding 1,2-RPP has $i$ columns with 1's and $n-i+1$ columns with 2's. Thus the lattice path gives the monomial $x_1^r x_2^{n-r+1}$ exactly when the edge lies in column $r$.

\begin{figure}[h]
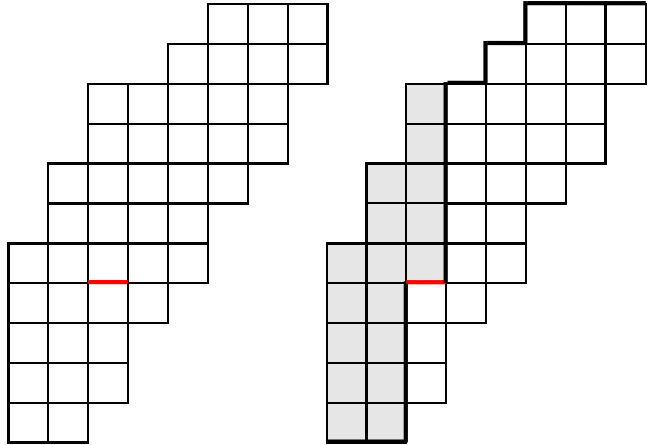

\[
\begin{ytableau}
\none & \none & \none & \none & \none &  &  &  \\
\none & \none & \none & \none & & & & \\
\none & \none & & & & & \\
\none & \none & & & & & \\
\none &  &    & & &  \\
\none &  &    & & \\
 	  &  &  \emptytikzmark{e}{1} & & \\
 	  &  &    & \\
 	  &  & \\
 	  &  & \\
 	  & 
\end{ytableau}
\DrawHLine[red, ultra thick]{e}{e}
\begin{ytableau}
\none & \none & \none & \none & \none & \emptytikzmark{13}{1} & \emptytikzmark{14}{1}  &  \emptytikzmark{15}{1} \\
\none & \none & \none & \none & \emptytikzmark{12}{1} & & & \\
\none & \none & *(mygray) \emptytikzmark{10}{1} & \emptytikzmark{11}{1} & & & \\
\none & \none & *(mygray)\emptytikzmark{9}{1} & & & & \\
\none 	  & *(mygray) & *(mygray)\emptytikzmark{8}{1}  & & &  \\
\none 	  & *(mygray) & *(mygray)\emptytikzmark{7}{1}   & & \\
*(mygray) & *(mygray) & *(mygray)\emptytikzmark{6}{1}   & & \\
*(mygray) & *(mygray)\emptytikzmark{5}{1} &    & \\
*(mygray) & *(mygray)\emptytikzmark{4}{1} & \\
*(mygray) & *(mygray)\emptytikzmark{3}{1} & \\
*(mygray) \emptytikzmark{1}{1} & *(mygray)\emptytikzmark{2}{1} 
\end{ytableau}
\]
\caption{Given an interior horizontal edge touching neither boundary, there is a unique lattice path with a single interior edge using the edge. If the edge lies in column $i$, the path contains $i$ columns with 1's and hence corresponds to the monomial $x_1^i x_2^{n-i+1}$.}\label{no_boundary}
\DrawHLine[black, ultra thick]{1}{2}
\DrawVLine[black, ultra thick]{2}{5}
\DrawHLine[red, ultra thick]{6}{6}
\DrawVLine[black, ultra thick]{6}{10}
\DrawHLineAbove[black, ultra thick]{11}{11}
\DrawVLineLeft[black, ultra thick]{12}{12}
\DrawHLineAbove[black, ultra thick]{12}{12}
\DrawVLineLeft[black, ultra thick]{13}{13}
\DrawHLineAbove[black, ultra thick]{13}{15}
\end{figure}

Next suppose the edge touches only the right boundary. Then as depicted in Figure \ref{right_boundary}, there may be multiple lattice paths using the edge: from the top right, the path may travel along the top boundary and drop down at any column before reaching column $i$. Note that the lattice path can correspond to a 1,2-RPP with between $i$ and $n$ columns containing 1's, and that the number of columns containing 1's determines the path. Similarly, if the edge touches only the left boundary then after reaching the edge the path can drop down to the bottom boundary at any column before $i$. Hence the lattice path can correspond to a 1,2-RPP with between $1$ and $i$ columns containing 1's.

\begin{figure}[h]
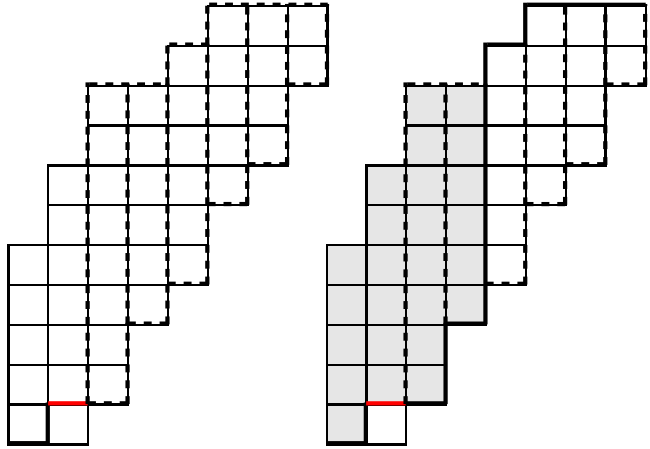

\[\begin{ytableau}
\none & \none & \none & \none & \none & \emptytikzmark{13}{1} & \emptytikzmark{14}{1}  &  \emptytikzmark{15}{1} \\
\none & \none & \none & \none & \emptytikzmark{12}{1} & & & \emptytikzmark{h}{1}\\
\none & \none & \emptytikzmark{10}{1} & \emptytikzmark{11}{1} & & & \\
\none & \none & \emptytikzmark{9}{1}  &  & & & \emptytikzmark{g}{1} \\
\none 	  &  & \emptytikzmark{8}{1}   &  &  & \emptytikzmark{f}{1}  \\
\none 	  &  & \emptytikzmark{7}{1}   &  &  \\
 &  & \emptytikzmark{6}{1}     &  & \emptytikzmark{e}{1} \\
 & \emptytikzmark{5}{1} & \emptytikzmark{c}{1}  & \emptytikzmark{d}{1} \\
 & \emptytikzmark{4}{1} & \emptytikzmark{b}{1}\\
 & \emptytikzmark{3}{1} & \emptytikzmark{a}{1} \\
 \emptytikzmark{1}{1} &  
\end{ytableau}
\DrawHLine[black, ultra thick]{1}{1}
\DrawVLine[black, ultra thick]{1}{1}
\DrawHLine[red, ultra thick]{3}{3}
\DrawHLine[black, dashed, ultra thick]{a}{a}
\DrawVLine[black, dashed, ultra thick]{a}{b}
\DrawHLine[black, dashed, ultra thick]{d}{d}
\DrawVLine[black, dashed, ultra thick]{d}{11}
\DrawVLineLeft[black, dashed, ultra thick]{12}{12}
\DrawHLineAbove[black, dashed, ultra thick]{12}{12}
\DrawVLineLeft[black, dashed, ultra thick]{13}{13}
\DrawHLineAbove[black, dashed, ultra thick]{13}{15}
\DrawVLine[black, dashed, ultra thick]{c}{10}
\DrawVLineLeft[black, dashed, ultra thick]{a}{10}
\DrawHLineAbove[black, dashed, ultra thick]{10}{11}
\DrawHLine[black, dashed, ultra thick]{e}{e}
\DrawVLine[black, dashed, ultra thick]{e}{12}
\DrawHLine[black, dashed, ultra thick]{f}{f}
\DrawVLine[black, dashed, ultra thick]{f}{13}
\DrawHLine[black, dashed, ultra thick]{g}{g}
\DrawVLine[black, dashed, ultra thick]{g}{14}
\DrawHLine[black, dashed, ultra thick]{h}{h}
\DrawVLine[black, dashed, ultra thick]{h}{15}
\begin{ytableau}
\none & \none & \none & \none & \none & \emptytikzmark{13}{1} & \emptytikzmark{14}{1}  &  \emptytikzmark{15}{1} \\
\none & \none & \none & \none & \emptytikzmark{12}{1} & & & \emptytikzmark{h}{1}\\
\none & \none & *(mygray) \emptytikzmark{10}{1} & *(mygray)\emptytikzmark{11}{1} & & & \\
\none & \none & *(mygray)\emptytikzmark{9}{1} & *(mygray) & & & \emptytikzmark{g}{1} \\
\none 	  & *(mygray) & *(mygray)\emptytikzmark{8}{1}     & *(mygray) &  & \emptytikzmark{f}{1}  \\
\none 	  & *(mygray) & *(mygray)\emptytikzmark{7}{1}     & *(mygray) &  \\
*(mygray) & *(mygray) & *(mygray)\emptytikzmark{6}{1}     & *(mygray) & \emptytikzmark{e}{1} \\
*(mygray) & *(mygray)\emptytikzmark{5}{1} & *(mygray)\emptytikzmark{c}{1}  & *(mygray)\emptytikzmark{d}{1} \\
*(mygray) & *(mygray)\emptytikzmark{4}{1} & *(mygray)\emptytikzmark{b}{1}\\
*(mygray) & *(mygray)\emptytikzmark{3}{1} & *(mygray)\emptytikzmark{a}{1} \\
*(mygray) \emptytikzmark{1}{1} &  
\end{ytableau} \]
\DrawHLine[black, ultra thick]{1}{1}
\DrawVLine[black, ultra thick]{1}{1}
\DrawHLine[red, ultra thick]{3}{3}
\DrawHLine[black, ultra thick]{a}{a}
\DrawVLine[black, ultra thick]{a}{b}
\DrawHLine[black, ultra thick]{d}{d}
\DrawVLine[black, ultra thick]{d}{11}
\DrawVLineLeft[black, ultra thick]{12}{12}
\DrawHLineAbove[black, ultra thick]{12}{12}
\DrawVLineLeft[black, ultra thick]{13}{13}
\DrawHLineAbove[black, ultra thick]{13}{15}
\DrawVLine[black, dashed, ultra thick]{c}{10}
\DrawVLineLeft[black, dashed, ultra thick]{a}{10}
\DrawHLineAbove[black, dashed, ultra thick]{10}{11}
\DrawHLine[black, dashed, ultra thick]{e}{e}
\DrawVLine[black, dashed, ultra thick]{e}{12}
\DrawHLine[black, dashed, ultra thick]{f}{f}
\DrawVLine[black, dashed, ultra thick]{f}{13}
\DrawHLine[black, dashed, ultra thick]{g}{g}
\DrawVLine[black, dashed, ultra thick]{g}{14}
\DrawHLine[black, dashed, ultra thick]{h}{h}
\DrawVLine[black, dashed, ultra thick]{h}{15}
\caption{When the edge touches only the right boundary, a lattice path using this edge can now drop down from the top boundary at any column after column $i$. However, there is a unique path corresponding to the monomial $x_1^r x_2^{n-r+1}$ if $i \leq r$ and no possible paths if $i > r$.}\label{right_boundary}
\end{figure}

Thus we have identified three cases where there is at least one lattice path corresponding to $x_1^r x_2^{n-r+1}$: the interior horizontal edge lies in one of column $1,\ldots,r$ and touches the right boundary, the edge lies in column $r$ and touches neither boundary, or the edges lies in one of column $r,\ldots,n$ and touches the left boundary. We will consider the fourth case, where the interior edge is a bottleneck edge, in the next paragraph. Fix two adjacent rows, and consider the set of horizontal edges between these two rows. The columns that these edges lie in are either all to the left of column $r$, contain column $r$, or all to the right of column $r$. In any of these three cases there is exactly one valid edge, as depicted in Figure \ref{possible_edges}. That is, between any two adjacent rows there is exactly one edge that correponds to at least one lattice path.  Since there are $m$ rows, this gives $m-1$ possible valid interior horizontal edges. Unless the edges are bottleneck edges, each possible edge corresponds to a single lattice path. It remains to count the additional lattice paths given by bottleneck edges.

\begin{figure}[h]
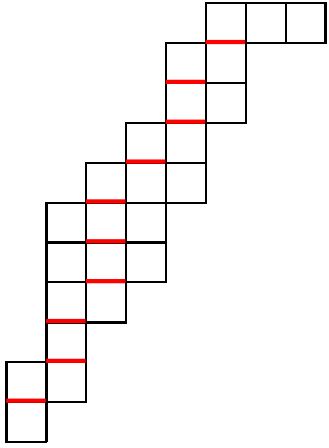

\begin{ytableau}
\none & \none & \none & \none & \none   & \emptytikzmark{10}{1} & & \\
\none & \none & \none & \none & \emptytikzmark{9}{1} & \\
\none & \none & \none & \none & \emptytikzmark{8}{1} & \\
\none & \none & \none   & \emptytikzmark{7}{1} &   \\
\none & \none & \emptytikzmark{6}{1} &  &  \\
\none &  &  \emptytikzmark{5}{1} &  \\
\none &  & \emptytikzmark{4}{1} &  \\
\none & \emptytikzmark{3}{1} &  \\
\none & \emptytikzmark{2}{1}&\none \\
\emptytikzmark{1}{1} &  \\
 & \none
\end{ytableau}

\DrawHLine[red, ultra thick]{1}{1}
\DrawHLine[red, ultra thick]{2}{2}
\DrawHLine[red, ultra thick]{3}{3}
\DrawHLine[red, ultra thick]{4}{4}
\DrawHLine[red, ultra thick]{5}{5}
\DrawHLine[red, ultra thick]{6}{6}
\DrawHLine[red, ultra thick]{7}{7}
\DrawHLine[red, ultra thick]{8}{8}
\DrawHLine[red, ultra thick]{9}{9}
\DrawHLine[red, ultra thick]{10}{10}

\caption{There are $m-1$ possible edges that can be chosen as the interior horizontal edge for a lattice path. Unless the edge is a bottleneck edge, each such edge corresponds to a unique lattice path.} \label{possible_edges}
\label{12RPPex}
\end{figure}

Now suppose the interior horizontal edge is a bottleneck edge lying in column $i$. Then there is flexibility on both sides: there can be between 0 and $(i-1)$ 1-pure columns to the left of column $i$ and between $0$ and $(n-i)$ 2-pure columns to the right of column $i$. If there are $x$ 1-pure columns to the left of column $i$, $x$ may be between 0 and $\max(i-1,r-1)$. If $x$ 1-pure columns lie to the left, the remaining $(r-x-1)$ 1-pure columns can be chosen to be to the right of column $i$ (because we assumed that $r \leq n-r+1$). Hence there are $\max(i,r)$ possible lattice paths using a given bottleneck edge in column $i$. 

We can now give a formula for the coefficient of $x_1^r x_2^{n-r+1}$. Let $k = \lceil\frac{n}{2}\rceil$ and $f_i= b_i + b_{n-i+1}$ for $i=1, 2, \ldots, k-1$. If $n$ is even, let $f_k = b_k + b_{n-k+1}$ and if $n$ is odd, let $f_k = b_k$. There are always at least $m-1$ valid lattice paths. Each bottleneck edge in column $i$ also contributes an additional $\max(i,r)-1$ lattice paths. Hence the coefficient is 
\[
(m - 1) + f_2 + 2f_3+3f_4+ \ldots + (r-1)f_r + (r-1)f_{r+1}+ \ldots + (r-1)f_k.
\]
Let $t_r$ denote the coefficient of $x_1^{r}x_2^{n-r+1}$. Note that for $2\leq{r}\leq{k-1}$, we have $2t_{r}-t_{r-1}-t_{r+1}=f_r$. Since any two $g$-equivalent shapes $\lambda/\mu$ and $\gamma/\nu$ must have the same coefficients $t_r$, it follows that for $2\leq{r}\leq{k-1}$ the sums $f_r = b_r + b_{n-r+1}$ are the same for the two shapes. Also, since
\[
t_k = (m - 1) + f_2 + 2f_3+3f_4+ \ldots + (k-1)f_k 
\]
is invariant for the two shapes, it then follows that $f_k$ is invariant for the two shapes.

By Corollary 8.11 in \cite{rsw2009coincidences}, we also have that $b_1 + \cdots + b_n = f_1 + \cdots + f_k$ is invariant, since the total number of bottleneck edges is the number of 1s in the 2-row overlap composition. Hence $f_1$ is invariant as well.
\end{proof}

\begin{remark}
For a ribbon $[\alpha_1, \alpha_2, \ldots, \alpha_n]$, we have $b_i = \alpha_i - 1$ for $i=1, 2, \ldots , n$. Hence Theorem \ref{bottleneck_cond} generalizes Lemma \ref{ribbon-col-sum} as noted at the beginning of the section.
\end{remark}

\begin{ex}
\ytableausetup{boxsize=.35cm}
It is noted in \cite{rsw2009coincidences} that the shapes \[\ydiagram{4+2,2+3,1+4,1+2,2,2}\qquad \ydiagram{4+2,3+2,3+2,1+3,4,2}\] are Schur equivalent. But since $b_2 + b_5 = 2$ for the first shape and $b_2+b_5 = 1$ for the second shape, it follows that the two shapes are not $g$-equivalent.
\end{ex}

\begin{ex} 
Having the same bottleneck edge sequence is not sufficient for two skew shapes to be $g$-equivalent. By Theorem $7.6$ in \cite{rsw2009coincidences}, the shapes $D_1$ and $D_2$ below are equivalent and have the same bottleneck edge sequence. However, upon computation it is found that they are not $g$-equivalent. 

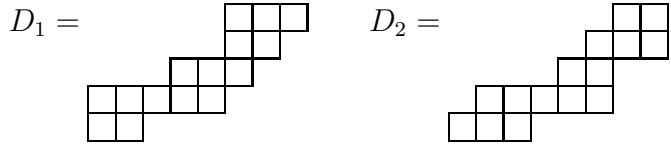
\[
D_1 = \ydiagram{5+3,5+2,3+3,5,2}\qquad
D_2 = \ydiagram{6+2,5+3,4+2,1+5,3}
\]
\end{ex}

Since the bottleneck condition followed as a result of comparing terms of $g$ with degree $n+1$ in two variables, it is natural to compute coefficients for terms of higher degree or more variables. The following result shows that terms of degree $n+1$ and more than two variables do not impose additional constraints for two skew shapes to be $g$-equivalent.

\begin{prop}
Suppose two skew shapes $\lambda/\mu$ and $\gamma/\nu$ have the same number of rows and the polynomial $g_{\lambda/\mu}$ and $g_{\gamma/\nu}$ have same coefficient for every term of degree $n+1$ with two variables. Then in fact these polynomials have the same coefficient for any term of degree $n+1$.
\end{prop}
\begin{proof}
Fix positive integers $i_1,i_2,\ldots,i_k$ where $k\geq2$ is some positive integer, and let $n = (\sum_{j=1}^k i_j)-1$. Given a skew diagram $\lambda/\mu$ with $n$ columns, we claim that the coefficient of $x_1^{i_1}x_2^{i_2}\ldots x_k^{i_k}$ can be expressed as a $\mathbb{Z}$-linear combination $c_0 + c_2b_2 + \ldots c_{n-1}b_{n-1}$ of the bottleneck numbers $b_i$. Furthermore, the constant $c_0$ is known to be $(k-1)(m-1)$, where $m$ is the number of rows in $\lambda/\mu$. We proceed by induction on $k$. 

The base case $k=2$ is given in the proof of Theorem \ref{bottleneck_cond}, so we may assume $k \geq 3$. We count the number of reverse plane partitions giving the monomial $x_1^{i_1}\cdots x_k^{i_k}$. 

Suppose first that every column containing a $1$ is in fact $1$-pure. Then the first $i_1$ columns must be filled with $1$'s. Note the remaining squares form a skew shape with $n-i_1$ columns, as depicted in Figure \ref{induction}. We henceforth use $(\lambda/\mu)_{i_1}$ to denote the skew shape given by removing the first $i_1$ columns of $\lambda/\mu$. Note $(\lambda/\mu)_{i_1}$ must be filled with a reverse plane partition giving the monomial $x_2^{i_2}\cdots x_k^{i_k}$.
\ytableausetup{boxsize=.5cm}
\begin{figure}[h]
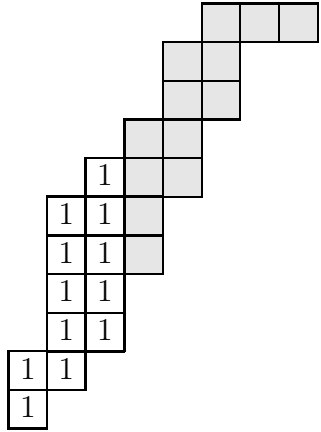

\begin{ytableau}
\none & \none & \none & \none & \none &*(mygray) &*(mygray) &*(mygray) \\
\none & \none & \none & \none & *(mygray) & *(mygray)\\
\none & \none & \none & \none & *(mygray) & *(mygray)\\
\none & \none & \none & *(mygray) & *(mygray) \\
\none & \none & 1 & *(mygray) & *(mygray) \\
\none & 1 &     1 & *(mygray) \\
\none & 1 &     1 & *(mygray) \\
\none & 1 &     1 \\
\none & 1 &     1 \\
1     & 1 \\
1
\end{ytableau}
\caption{The remaining shape shaded in gray is a skew shape with $n-i_1$ columns, denoted $(\lambda/\mu)_{i_1}$.}\label{induction}
\end{figure}

Let $m'$ be the number of rows in the shape obtained by removing the first $i_1$ columns from $\lambda/\mu$. Then by induction the number of ways to fill in this shape is $$(k-2)(m'-1) + c'_{i_1+2} b_{i_1+2} + \cdots c'_{n-1} b_{n-1}$$ for some integers $c'_{i_1+2},\ldots,c'_{n-1}.$

The remaining case is when the reverse plane partition has a mixed column containing a 1. Given such a reverse plane partition, consider the 1,2-RPP obtained by replacing every entry greater than or equal to 2 with 2. From this 1,2-RPP we obtain a lattice path via our previously described bijection between 1,2-RPPs and lattice paths. Since the reverse plane partition has a mixed column containing a 1 and the total degree of the monomial $x_1^{i_1}x_2^{i_2}\ldots x_k^{i_k}$ is $n+1$, this lattice path must have a single interior horizontal edge. As noted in Figure \ref{12RPPex}, there are $m-1$ possibilities for the unique interior horizontal edge. 

Consider first the interior horizontal edges in columns $1,\ldots,i_1$ touching the bottom boundary of $\lambda/\mu$. This case is depicted in Figure \ref{case_1}. Note that there are $m-m'$ such edges, since in total there are $m-1$ edges touching the bottom boundary and exactly $m'-1$ of them lie in columns $i_1+1,\ldots,n$. For each of these $m-m'$ edges, one possible lattice path. The path starts at the top right, travels along the top boundary until it reaches the boundary between column $i_1$ and column $i_1+1$, drops down to the bottom boundary, and travels along the bottom boundary until the horizontal edge, traverses the edge, and then immediately drops back down to the bottom boundary and traverses it until reaching the bottom left. This lattice path determines which squares are filled with 1's. The remaining shape is a disconnected skew shape where one component is a single column and the other component is $(\lambda/\mu)_{i_1}$. There are $(k-1)$ fillings using this lattice path, since the column below the edge may be filled with any of $2,\ldots,k$ and the remaining columns must fill $(\lambda/\mu)_{i_1}$ in increasing order. Note that unless the edge is a bottleneck edge, this is the unique lattice path using this edge.

\begin{figure}[h]
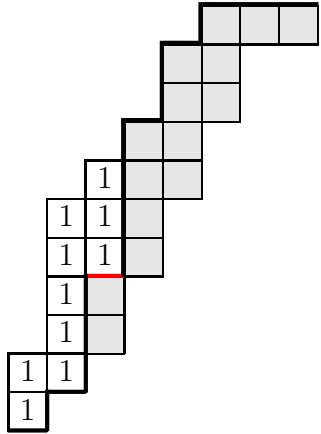

\begin{ytableau}
\none & \none & \none & \none & \none &*(mygray)\emptytikzmark{10}{1} &*(mygray)\emptytikzmark{11}{1} &*(mygray)\emptytikzmark{12}{1} \\
\none & \none & \none & \none & *(mygray)\emptytikzmark{9}{1} & *(mygray)\\
\none & \none & \none & \none & *(mygray)\emptytikzmark{8}{1} & *(mygray)\\
\none & \none & \none & *(mygray)\emptytikzmark{7}{1} & *(mygray) \\
\none & \none & 1 & *(mygray) & *(mygray) \\
\none & 1 &     1 & *(mygray) \\
\none & 1 &     \tikzmark{5}{1} & *(mygray)\emptytikzmark{6}{1} \\
\none & \tikzmark{4}{1} &     *(mygray) \\
\none & \tikzmark{3}{1} &     *(mygray) \\
1     & \tikzmark{2}{1} \\
\tikzmark{1}{1}
\end{ytableau}
\caption{Case 1: the lattice path uses an edge in columns $1,\ldots,i_1$ touching the bottom boundary. Then the remaining shape is the union of $(\lambda/\mu)_{i_1}$ and a single column.}\label{case_1}
\DrawHLine[red, ultra thick]{5}{5}
\DrawHLine[black, ultra thick]{1}{1}
\DrawHLine[black, ultra thick]{2}{2}
\DrawVLine[black, ultra thick]{1}{1}
\DrawVLine[black, ultra thick]{2}{4}
\DrawHLineAbove[black, ultra thick]{7}{7}
\DrawVLineLeft[black, ultra thick]{6}{7}
\DrawVLineLeft[black, ultra thick]{8}{9}
\DrawHLineAbove[black, ultra thick]{9}{9}
\DrawVLineLeft[black, ultra thick]{10}{10}
\DrawHLineAbove[black, ultra thick]{10}{12}
\end{figure}

The remaining $m'-1$ edges are those in column $i_1$ not touching the bottom boundary and the edges in column $i_i+1,\ldots,n$ touching the top boundary. We similarly describe a possible lattice path for each of these edges. Suppose the edge lies in column $i$. The path starts at the top right, travels along the top boundary until the boundary between column $i$ and $i+1$, drops down to the edge and traverses it, traverses the top boundary until the boundary between column $i_1-1$ and column $i_1$, and drops down to the bottom boundary and traverses it until reaching the bottom left. This path determines which squares are filled with 1's. The remaining squares form a (possibly disconnected) skew shape, which must be filled with no mixed columns. Note the remaining skew shape is connected if and only if the horizontal edge was not a bottleneck edge. If the shape is connected, then filling the columns in increasing order is the only possible filling. Otherwise, this is one possible filling but there may be more.

Thus far, this gives us \begin{eqnarray*} & &(k-2)(m'-1) + c'_{i_1+2} b_{i_1+2} + \cdots c'_{n-1} b_{n-1} + (k-1)(m-m') + (m'-1) \\ &=& (k-1)(m-1) + c'_{i_1+2} b_{i_1+2} + \cdots c'_{n-1} b_{n-1}\end{eqnarray*} fillings. It remains to show each bottleneck edge in column $i$ contributes a fixed number of additional fillings depending only on $i$.
 \begin{figure}[h!]
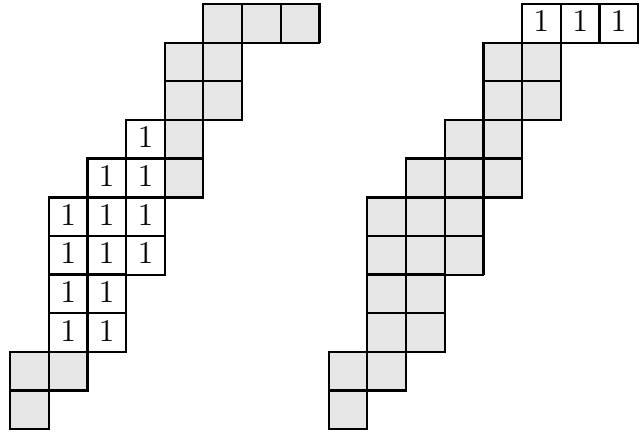

\begin{ytableau}
\none & \none & \none & \none & \none &*(mygray) &*(mygray) &*(mygray) \\
\none & \none & \none & \none & *(mygray) & *(mygray) \\
\none & \none & \none & \none & *(mygray) & *(mygray) \\
\none & \none & \none & 1 &  *(mygray)\\
\none & \none & 1 & 1 & *(mygray) \\
\none & 1 &     1 & 1 \\
\none & 1 &     1 & 1 \\
\none & 1 &     1 \\
\none & 1 &     1 \\
 *(mygray)    & *(mygray) \\
*(mygray)     & \none
\end{ytableau}
\begin{ytableau}
\none & \none & \none & \none & \none & 1 & 1 & 1 \\
\none & \none & \none & \none & *(mygray) & *(mygray) \\
\none & \none & \none & \none & *(mygray) & *(mygray) \\
\none & \none & \none & *(mygray) &  *(mygray)\\
\none & \none & *(mygray) & *(mygray) & *(mygray) \\
\none & *(mygray) & *(mygray) & *(mygray) \\
\none & *(mygray) & *(mygray) & *(mygray) \\
\none & *(mygray) & *(mygray) \\
\none & *(mygray) & *(mygray) \\
 *(mygray)    & *(mygray) \\
*(mygray)     & \none
\end{ytableau}
\caption{The remaining shape will have 1 or 2 components. The number of fillings is determined by $i_2,\ldots,i_k$ and the number of columns in the components.}\label{cases}
\end{figure}

As noted in the proof of Theorem \ref{bottleneck_cond}, each bottleneck edge in column $i$ has $\min(i,n-i+1,i_1)$ possible lattice paths using that edge. Each lattice path determines which squares will be filled with 1's. Note the remaining squares will form a possibly disconnected skew shape with $n-i_1+1$ columns (depicted in Figure \ref{cases}), which must then be filled with no mixed columns. There are a fixed number of ways to fill this shape, which depends only on $i_2,\ldots,i_k$ and the number of columns in the two components. The possible number of columns in each component is in turn determined by which column the bottleneck edge is in. This finishes the proof of the claim.

 \begin{figure}[h!]
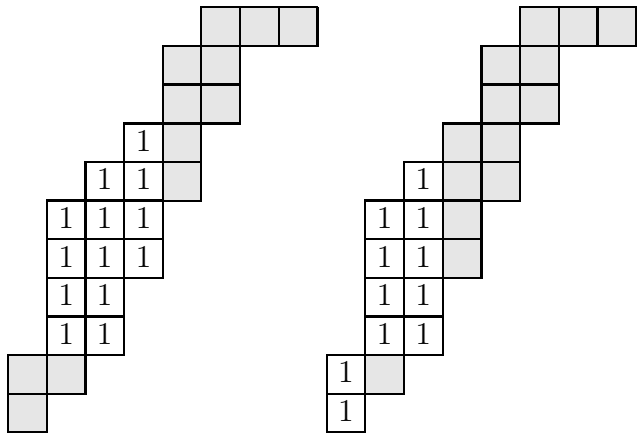

\begin{ytableau}
\none & \none & \none & \none & \none &*(mygray) &*(mygray) &*(mygray) \\
\none & \none & \none & \none & *(mygray) & *(mygray) \\
\none & \none & \none & \none & *(mygray) & *(mygray) \\
\none & \none & \none & 1 &  *(mygray)\\
\none & \none & 1 & 1 & *(mygray) \\
\none & 1 &     1 & 1 \\
\none & 1 &     1 & 1 \\
\none & 1 &     1 \\
\none & 1 &     1 \\
 *(mygray)    & *(mygray) \\
*(mygray)     & \none
\end{ytableau}
\begin{ytableau}
\none & \none & \none & \none & \none &*(mygray) &*(mygray) &*(mygray) \\
\none & \none & \none & \none & *(mygray) & *(mygray) \\
\none & \none & \none & \none & *(mygray) & *(mygray) \\
\none & \none & \none & *(mygray) &  *(mygray)\\
\none & \none & 1 & *(mygray) & *(mygray) \\
\none & 1 &     1 & *(mygray) \\
\none & 1 &     1 & *(mygray) \\
\none & 1 &     1 \\
\none & 1 &     1 \\
1    & *(mygray) \\
1     & \none
\end{ytableau}
\caption{The possible numbers of columns in the components of the remaining shape are determined by which column the bottleneck edge is in. In this example, since the bottleneck edge is in column 2, the components have 1 and $n-i_1+1$ columns or 2 and $n-i_1+2$ columns.}
\end{figure}
Thus we have that the coefficient of $x_1^{i_1}\cdots x_k^{i_k}$ for any shape with $n = (\sum_{j_1}^k i_j) - 1$ columns is $(k-1)(m-1) + c_2 b_1 + \cdots + c_{n-1} b_{n-1}$ for some integers $c_2,\ldots,c_{n-1}$. Recall that every shape is equivalent to its 180-degree rotation, and note 180-degree rotation reverses the bottleneck sequence $b_1,\ldots,b_n$. Since there are shapes with arbitrary sequences of $b_1,\ldots,b_n$ (for example, the ribbon $[b_1+1,\ldots,b_n+1]$), it follows that $c_i = c_{n-i+1}$ for $i=2,\ldots,n-1$. Recall also that the proof of Theorem \ref{bottleneck_cond} shows each sum $b_i + b_{n-i+1}$ for $i=2,\ldots,n-1$ must be the same for any two shapes such that the terms in $g$ of degree $n+1$ with two variables are the same. Since the number of rows $m$ must be the same as well, it follows that the sum $(k-1)(m-1) + c_2 b_1 + \cdots + c_{n-1} b_{n-1}$ must also be the same.
\end{proof}

\begin{prop} \label{x1_2 x2_n}
The coefficient of $x_1^2 x_2^n$ in $g_{\lambda/\mu}$ is 
\[
\binom{m}{2} - \sum_{i=1}^n \binom{b_i+1}{2} .
\]
\end{prop}
\begin{proof}
A 1,2-RPP giving the monomial $x_1^2 x_2^n$ must have no 1-pure columns, $(n-2)$ 2-pure columns, and two mixed columns. Hence the corresponding lattice paths have two interior horizontal edges. Consider the \textit{heights} of the interior horizontal edges. By an interior horizontal edge at height $i$ we mean the edge lies between row $i$ and row $i+1$. Observe that given the height of the two interior horizontal edges, there is at most one lattice path using the heights; since there are no 1-pure columns, the lattice path is completely determined by the heights chosen. 
\ytableausetup{boxsize=.5cm}
\begin{figure}[h]
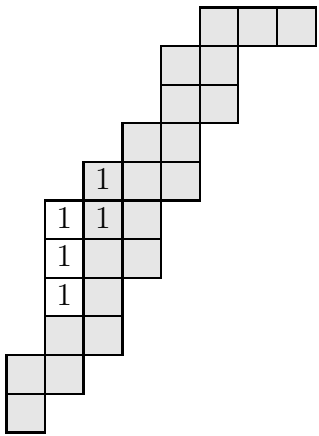

\begin{ytableau}
\none & \none & \none & \none & \none &*(mygray) &*(mygray) &*(mygray) \\
\none & \none & \none & \none & *(mygray) & *(mygray)\\
\none & \none & \none & \none & *(mygray) & *(mygray) \\
\none & \none & \none & *(mygray) & *(mygray) \\
\none & \none & *(mygray) 1 & *(mygray) & *(mygray) \\
\none & 1 &     *(mygray) 1 & *(mygray) \\
\none & 1 & *(mygray)     & *(mygray) \\
\none & 1 & *(mygray)     \\
\none & *(mygray) & *(mygray)     \\
*(mygray)     & *(mygray) \\
*(mygray)  
\end{ytableau}
\caption{The heights determine the filling, since the column containing 1's more to the left must touch the left boundary of the shape, and the other column must touch the left boundary of the remaining shape.}
\end{figure}

There are $\binom{m}{2}$ ways to choose a pair of heights from $1,\ldots,m-1$ (with possible repetition). Since each pair of heights contributes either 1 or 0 lattice paths, the desired coefficient is thus $\binom{m}{2}$ minus the number of pairs not giving a lattice path. These are exactly the pairs of heights where the only interior horizontal edges at those heights lie in a single column. This is precisely the pairs of bottleneck edges from the same column. For each column $i$ there are $\binom{b_i+1}{2}$ ways to choose two of the bottlenecks in column $i$ (with possible repetition), giving the desired formula.
\end{proof}

By Corollary 8.11 in \cite{rsw2009coincidences}, the number of rows $m$ and the sum $b_1+\cdots+b_n$ are invariant under $g$-equivalence. Hence we attain the following as a direct consequence of Proposition~\ref{x1_2 x2_n}.
\begin{cor}\label{cor:trianglesums}
Suppose $g_{\lambda/\mu} = g_{\gamma/\nu}$. Then
\[
\sum_{i=1}^n (b_i^{\lambda/\mu})^2 = \sum_{i=1}^n (b_i^{\gamma/\nu})^2.
\]
Equivalently, the sums of the areas of the equilateral triangles of 1's in the row overlap compositions $r^{(2)},\ldots,r^{(m)}$ are the same.
\end{cor}

\begin{remark}
One can also count various other coefficients in the dual stable Grothendieck polynomial. For terms of degree greater than $n+1$, it is useful to define a generalization of bottleneck edges. To that end, for $i = 1,\ldots,\lambda_1-w+1$ the \textit{number of width $w$ bottlenecks in position i} is
\[
b^{(w)}_i = |\{ 1 \leq j \leq m-1 \mid \mu_j = i-1, \lambda_{j+1} = i+w-1 \}|.
\] 
For example, let $\lambda / \mu = (5,5,4,2,2,2)/(4,2,1,1,1,0)$. Then the number of bottleneck edges of each width is given below. Note $b^{(1)}$ is just the previously defined bottleneck edges.
\[\begin{tabular}{>{$}l<{$\hspace{12pt}}*{13}{c}}
b^{(5)} &&&&&&&0&&&&&&\\
b^{(4)} &&&&&&0&&0&&&&&\\
b^{(3)} &&&&&0&&0&&0&&&&\\
b^{(2)} &&&&0&&0&&1&&0&&&\\
b^{(1)} &&&0&&3&&0&&0&&1&&\\
\end{tabular}\]

We state the following propositions with proofs omitted for brevity.

\begin{prop}\label{x1_3 x2_n-1}
The coefficient of $x_1^3 x_2^{n-1}$ in $g_{\lambda/\mu}$ is 
\[
\left( \binom{m}{2} - \sum_{i=1}^{n} \binom{b^{(1)}_i+1}{2} \right) + \sum_{i=2}^{n-2} \binom{b_i^{(2)}+1}{2} + (m-2)\sum_{i=2}^{n-1} b_i^{(1)}\]\[ - 
 \left(b_2^{(1)}(m-\mu_1'-1) + b_{n-1}^{(1)}(\lambda_n'-1) + \sum_{i=2}^{n-2} b_i^{(1)}b_{i+1}^{(1)} \right).
\]
\end{prop}
\begin{prop}
The coefficient of $x_1^3 x_2^n$ in $g_{\lambda/\mu}$ is 
\[
\binom{m+1}{3} - \sum_{i=1}^{n} \left((m-1)\binom{b_i^{(1)}+1}{2} - 2 \binom{b_i^{(1)}}{3} - b_i^{(1)}(b_i^{(1)}-1) \right) \]\[ - \sum_{i=1}^{n-1} \left(\binom{b_i^{(2)}+2}{3} + (b_i^{(1)}+b_{i+1}^{(1)})\binom{b_i^{(2)}+1}{2} + b_i^{(1)}b_i^{(2)}b_{i+1}^{(1)}\right).
\]
\end{prop}

\end{remark}

\section{Relation Between $g$-equivalence and $G$-equivalence}

It is natural to ask whether $g_A = g_B$ for two skew shapes $A$ and $B$ implies $G_A = G_B$, and vice versa. The following examples show that in general, neither equality implies the other.

\begin{ex}
\ytableausetup{boxsize=.35cm}
Based on computer computation, the shapes \[\ydiagram{4+4,1+5,4,2}\qquad \ydiagram{3+5,2+4,4,2}\] are $g$-equivalent but not $G$-equivalent. For example, the coefficients of $x_1^6 x_2^6 x_3^3 x_4$ in $G$ are $-353$ and $-354$, respectively.
\end{ex}

\begin{ex}\label{ribbon_staircase}
The shapes \[\ydiagram{3+5,3+3,1+3,2} \qquad \ydiagram{5+3,1+5,1+3,2}\] are $G$-equivalent but not $g$-equivalent. One can show $G$-equivalence through computer computation using the reverse lattice word expansion of $G_{\lambda/\mu}$ into stable Grothendieck polynomials indexed by straight shapes found in \cite{buch2002lrrule}. To see the shapes are not $g$-equivalent, we notice that $b_4+b_5 = 1$ for the shape on the left and $b_4+b_5=0$ for the shape on the right.
\end{ex}

\section{Future Explorations}
\subsection{Coincidences of ribbon stable Grothendieck polynomials}

The combinatorics of ribbon stable Grothendieck polynomials seem to be more difficult than their dual stable Grothendieck and Schur counterparts. However, we still conjecture that coincidences among ribbon Grothendieck polynomials arise in precisely the same way as the dual case. 

\begin{conj}
Let $\alpha$ and $\beta$ be ribbons. Then $G_\alpha = G_\beta$ if and only if $\beta=\alpha$ or $\beta=\alpha^*$. 
\end{conj}
While one direction is immediate, the other direction has proven to be much more difficult.

\subsection{Conjugation invariance}
Given a Young diagram $\lambda=\langle\lambda_1,\lambda_2,\ldots,\lambda_k\rangle$, we define its \textit{transpose} Young diagram to be $\lambda^T=\langle \lambda_1',\dots,\lambda_s'\rangle$, where $\lambda_i'$ is the number of boxes in column $i$ of $\lambda$. This operation extends to skew diagrams by setting $(\lambda/\mu ) ^T = \lambda^T/\mu^T$. For example, $\langle 5,5,2\rangle^T=\langle 3,3,2,2,2\rangle$ and $(\langle 4,3,1\rangle/\langle 2\rangle)^T=\langle 4,3,1\rangle^T/\langle 2\rangle^T=\langle 3,2,2,1\rangle/\langle 1,1\rangle$.

For skew shapes $A$ and $B$ it follows immediately from the Jacobi-Trudi identity that $s_A = s_B$ implies $s_{A^T} = s_{B^T}$. Since it remains open to find an analogue of Jacobi-Trudi for skew Grothendieck polynomials, the answer to the analogous question for $g$ and $G$ is less obvious.

\begin{question}
Suppose $g_A = g_B$. Does it follow that $g_{A^T} = g_{B^T}$? Similarly, suppose $G_A = G_B$. Does it follow that $G_{A^T} = G_{B^T}$?
\end{question}

If conjugation does preserve $g$-equivalence, then we immediately get another necessary condition on $g$-equivalence by taking a transposed version of Theorem \ref{bottleneck_cond}.

\subsection{Ribbon staircases}
Theorem 7.30 of \cite{rsw2009coincidences} describes a class of nontrivial skew equivalences. A \textit{nesting} is a word consisting of the symbols  left parenthesis ``$($," right parenthesis ``$)$," dot ``$.$," and vertical slash ``$|$" where the parentheses must be properly matched. Given a skew shape that may be decomposed into a ribbon $\alpha$ in a certain manner as described in \cite{rsw2009coincidences}, one may obtain a corresponding nesting. Theorem 7.30 states that shapes that may be decomposed with the same ribbon $\alpha$ such that the nestings are reverses of each other are Schur equivalent.

It is interesting to consider whether these equivalences hold for $g$ and $G$ as well. For example, Corollary 7.32 of \cite{rsw2009coincidences} states that $s_{\delta_n/\mu} = s_{(\delta_n/\mu)^T}$ for any diagram $\mu$ contained in the staircase partition $\delta_n = \langle n-1,n-2,\ldots,1\rangle $. Computation strongly suggests the same holds true for the Grothendieck polynomials as well.

\begin{conjecture}\label{stair_pair}
Let $\mu$ be a diagram contained in the staircase partition $\delta_n = \langle n-1,n-2,\ldots,1 \rangle$. Then $g_{\delta_n/\mu} = g_{(\delta_n/\mu)^T}$ and $G_{\delta_n/\mu} = G_{(\delta_n/\mu)^T}$.
\end{conjecture}

However, not all equivalences described by Theorem 7.30 hold for Grothendieck polynomials. 

\begin{question}
For which ribbons $\alpha$ and nestings $\mathcal{N}$ are the corresponding shapes $g$-equivalent or $G$-equivalent?
\end{question}

\section{Acknowledgments}
This research was carried out as part of the 2016 summer REU program at the University of Minnesota, Twin Cities and was supported by NSF RTG grant DMS-1148634 and by NSF grant DMS-1351590. We would like to thank Vic Reiner, Gregg Musiker, Sunita Chepuri, and Pasha Pylyavskyy for their mentorship and support. 

\bibliographystyle{alpha} 
\bibliography{main}
\end{document}